\documentclass[12pt, reqno]{amsart}
\usepackage{amsmath, amsthm, amscd, amsfonts, amssymb, graphicx}
\usepackage[bookmarksnumbered, colorlinks,plainpages]{hyperref}
\input{mathrsfs.sty}

\newtheorem{theorem}{Theorem}[section]
\newtheorem{lemma}[theorem]{Lemma}

\newtheorem{corollary}[theorem]{Corollary}
\theoremstyle{definition}

\theoremstyle{remark}
\newtheorem{remark}[theorem]{Remark}

\numberwithin{equation}{section}
\begin{document}

\title[Bessel type inequalities]{Bessel type inequalities in Hilbert $C^*$-modules}

\author[S.S. Dragomir, M. Khosravi, M.S. Moslehian]{S. S. Dragomir$^1$, M. Khosravi$^2$ and M. S. Moslehian$^3$}
\address{$^1$ Research Group in Mathematical Inequalities and Applications, School of  Engineering \& Science, Victoria University, P.
O. Box 14428, Melbourne city, Vic, 8001, Australia.}
\email{sever.dragomir@vu.edu.au}
\address{$^2$ Department of Mathematics, Tehran Teacher Training University,
P. O. Box 15618, Tahran, Iran\newline Banach Mathematical Research
Group (BMRG), Mashhad, Iran.} \email{khosravi$_-$m@saba.tmu.ac.ir}
\address{$^3$ Department of Mathematics, Ferdowsi University of Mashhad, P. O. Box
1159, Mashhad 91775, Iran;
\newline Center of Excellence in Analysis on
Algebraic Structures (CEAAS), Ferdowsi University of Mashhad, Iran.}
\email{moslehian@ferdowsi.um.ac.ir and moslehian@ams.org}

\begin{abstract}
Regarding the generalizations of the Bessel inequality in Hilbert
spaces which are due to Bombiari and Boas--Bellman, we obtain a
version of the Bessel inequality and some generalizations of this
inequality in the framework of Hilbert $C^*$-modules.
\end{abstract}

\maketitle

%--------------------------------------------------------------------------------------------------%

\section{Introduction}

The Bessel inequality states that if $(e_i)_{1\leq i\leq n}$ is an
orthonormal sequence of vectors in a Hilbert space $({\mathscr
H};\langle\cdot,\cdot\rangle)$, then
$$\sum_{i=1}^n\left|\langle x,e_i\rangle\right|^2\leq\|x\|^2 \qquad (x\in {\mathscr H})\,.$$
\noindent A number of mathematicians have investigated the above
inequality in various settings. One of the generalizations of the
Bessel inequality was given by Bombieri \cite{bom} as follows.
\begin{theorem}
If $x, y_1,\cdots,y_n$ are elements of a complex unitary space, then
$$\sum_{i=1}^n \left|\langle x,y_i\rangle\right|^2\leq\|x\|^2\max_{1\leq i\leq n}
\sum_{j=1}^n|\langle y_i,y_j\rangle|.$$
\end{theorem}

\noindent In 1941, Boas \cite{boas} and in 1944, independently,
Bellman \cite{bellman} proved a result, which can be stated as
follows.
\begin{theorem}
If $x,y_1,\cdots,y_n$ are elements of a Hilbert space, then
\begin{eqnarray*}
\sum_{i=1}^n|\langle x,y_i\rangle|^2\leq\|x\|^2\left[\max_{1\leq
i\leq n}\|y_i\|^2+\left( \sum_{1\leq i\neq j\leq n}|\langle
y_i,y_j\rangle|^2\right)^{1/2}\right].
\end{eqnarray*}
\end{theorem}
\noindent Recently, Mitrinovi\'{c}, Pe\v{c}ari\'{c} and Fink
\cite{pec} proved the following inequality and showed that it is
equivalent with the Boas--Bellman theorem.
\begin{theorem}
If $x,y_1,\cdots,y_n$ are elements of a Hilbert space and
$c_1,\cdots,c_n$ are arbitrary complex numbers, then
$$|\sum_{i=1}^nc_i\langle x,y_i\rangle|^2\leq\|x\|^2\sum_{i=1}^n|c_i|^2
\left[\max_{1\leq i\leq n}\|y_i\|^2+ \left(\sum_{1\leq i\neq j\leq
n}|\langle y_i,y_j\rangle|^2\right)^{1/2}\right].$$
\end{theorem}
\noindent In addition, Dragomir \cite{drabook} obtained some other
generalizations of the Bessel inequality, which are similar to
Boas--Bellman inequality and Mitrinovi\'{c}--Pe\v{c}ari\'{c}--Fink
inequality.

Our aim is to extend some of these generalizations of the Bessel
inequality in the framework of Hilbert $C^*$-modules. A related
concept to our work is the notion of frame. We would like to refer
the interested reader to \cite{F-L} for an extensive account on
frames in Hilbert $C^*$-modules.

\section{Preliminaries}
In this section we recall some fundamental definitions in the theory
of Hilbert modules that will be used in the sequel.

\noindent Suppose that $\mathscr{A}$ is a $C^*$-algebra and
$\mathscr{X}$ is a linear space which is an algebraic right
$\mathscr{A}$-module satisfying $\lambda(xa)=x(\lambda a)=(\lambda
x)a$ for all $x \in {\mathscr X},a \in {\mathscr A}, \lambda \in
{\mathbb C}$. The space $\mathscr{X}$ is called a
\textit{pre-Hilbert $\mathscr{A}$-module} (or an \textit{inner
product $\mathscr{A}$-module}) if there exists an
\textit{$\mathscr{A}$-valued inner product}
$\langle\cdot,\cdot\rangle:\mathscr{X}\times \mathscr{X}\to
\mathscr{A}$ satisfying the following properties:\\
(i) $\langle x,x\rangle\geq0$ and $\langle x,x\rangle=0$ if and
only if $x=0$;\\
(ii) $\langle x,\lambda y+z\rangle=\lambda\langle
x,y\rangle+\langle x,z\rangle$;\\
(iii) $\langle x,ya\rangle=\langle x,y\rangle a$;\\
(iv) $\langle x,y\rangle^*=\langle y,x\rangle$;\\
for all $x,y,z\in \mathscr{X}$, $a\in \mathscr{A}$ and
$\lambda\in\mathbb{C}$.

\noindent By (ii) and (iv), $\langle \lambda x+ya, z\rangle=
\bar\lambda \langle x, z\rangle + a^*\langle y, z\rangle$. It
follows from the \textit{Cauchy--Schwarz inequality} $\langle
y,x\rangle\langle x,y\rangle\leq\|\langle x,x\rangle\| \langle
y,y\rangle$ that $\|x\|=\|\langle x,x\rangle\|^{\frac{1}{2}}$ is a
norm on $\mathscr{X}$, where the latter norm denotes that in the
$C^*$-algebra ${\mathscr A}$; see \cite[p. 5]{lance}. This norm
makes $\mathscr{X}$ into a right normed module over $\mathscr{A}$. A
pre-Hilbert module $X$ is called a \textit{Hilbert
$\mathscr{A}$-module} if it is complete with respect to its norm.

\noindent Two typical examples of Hilbert $C^*$-modules are as
follows:

(I) Every Hilbert space is a Hilbert ${\mathbb C}$-module.

(II) Every $C^*$-algebra ${\mathscr A}$ is a Hilbert ${\mathscr
A}$-module via $\langle a, b\rangle = a^*b \quad (a, b \in {\mathscr
A})$.

\noindent Notice that the inner product structure of a $C^*$-algebra
is essentially more complicated than complex numbers. For instance,
the concepts such as adjoint, orthogonality and theorems such as
Riesz' representation in the theory of complex Hilbert spaces cannot
simply be generalized or transferred to that of Hilbert
$C^*$-modules.

\noindent One may define an ``\textit{$\mathscr{A}$-valued norm}''
$|.|$ by $|x|=\langle x,x\rangle^{1/2}$. Clearly,
$\|\,|x|\,\|=\|x\|$ for each $x\in \mathscr{X}$. It is known that
$|.|$ does not satisfy the triangle inequality in general. See
\cite{lance, M-T} for more detailed information on Hilbert
$C^*$-modules.

%--------------------------------------------------------------------------------------------------%
\section{main results}

We start our work by presenting a version of the Bessel inequality
for Hilbert $C^*$-modules.
%--------------------------------------------------------------------------------------------------%

\begin{theorem}
Let $\mathscr{X}$ be a Hilbert $\mathscr{A}$-module and
$e_1,e_2,\cdots, e_n$ be a sequence of unit vectors in $X$ such that
$\langle e_i,e_j\rangle=0$ for $1 \leq i\neq j \leq n$. If $x\in X$,
then
\begin{eqnarray}\label{bessel}\sum_{i=1}^n|\langle
e_i,x\rangle|^2 \leq |x|^2.\end{eqnarray}
\end{theorem}
\begin{proof}
The result follows from the following inequalities.
\begin{eqnarray*}
0\leq \left|x-\sum\limits_{i=1}^n e_i\langle
e_i,x\rangle\right|^2&=&\left\langle x-\sum\limits_{i=1}^n
e_i\langle
e_i,x\rangle ,x-\sum\limits_{i=1}^n e_i\langle e_i,x\rangle \right\rangle\\
&=&\langle x,x\rangle+\sum\limits_{i=1}^n\sum\limits_{j=1}^n\langle
e_i,x\rangle^*\langle
e_i,e_j\rangle\langle e_j,x\rangle-2\sum\limits_{i=1}^n|\langle e_i,x\rangle|^2\\
&=&\langle x,x\rangle+\sum\limits_{i=1}^n\langle
e_i,x\rangle^*\langle
e_i,e_i\rangle\langle e_i,x\rangle-2\sum\limits_{i=1}^n|\langle e_i,x\rangle|^2\\
&\leq& |x|^2+\sum\limits_{i=1}^n\langle e_i,x\rangle^*\langle
e_i,x\rangle-2\sum\limits_{i=1}^n|\langle
e_i,x\rangle|^2\\
&=&|x|^2-\sum\limits_{i=1}^n|\langle e_i,x\rangle|^2.
\end{eqnarray*}
\end{proof}
%-------------------------------------------------------------------------------------

The following lemma is useful to prove a Bombieri type inequality.
%-----------------------------------------------------------------------------
\begin{lemma}\label{8}
Let ${\mathscr A}$ be a $C^*$-algebra and $a,b,c \in {\mathscr A}$.
Then
$$a^*cb+b^*c^*a\leq\|c\|(|a|^2+|b|^2).$$
\end{lemma}
\begin{proof}
Using the universal representation of ${\mathscr A}$ it is
sufficient to prove that $$T^*RS+S^*R^*T\leq\|R\|(|T|^2+|S|^2)$$ for
bounded linear operators $T,R,S$ acting on a Hilbert space
${\mathscr H}$.

By the polar decomposition, there exists a partial isometry $U\in
{\mathbb B}(\mathscr {H})$ such that $R=U|R|$. It follows from
\begin{eqnarray*}
0&\leq& (S-U^*T)^*|R|(S-U^*T)\\
&=&S^*|R|S-S^*|R|U^*T-T^*U|R|S+ T^*U|R|U^*T,
\end{eqnarray*}
that
$$T^*RS+S^*R^*T\leq |\,|R|^{1/2}U^*T|^2+|\,|R|^{1/2}S|^2.$$ Hence
\begin{eqnarray*}
T^*RS+S^*R^*T&\leq&
\left|\,|R|^{1/2}U^*T\right|^2+\left|\,|R|^{1/2}S\right|^2\\
&\leq&\left\|\,|R|^{1/2}U^*\right\|^2|T|^2+\left\|\,|R|^{1/2}\right\|^2\,|S|^2\\
&\leq&\|R\|(|T|^2+|S|^2).
\end{eqnarray*}
\end{proof}
%-------------------------------------------------------------------------
\begin{theorem}\label{main}
If $y_1,\cdots, y_n$ are elements of a Hilbert $C^*$-module
$\mathscr{X}$, then
\begin{eqnarray}\label{eq}
\left|\sum_{i=1}^ny_ia_i\right|^2\leq \sum_{i=1}^n|a_i|^2\max_{1\leq
i\leq n}\sum_{j=1}^n \|\langle y_i,y_j\rangle\|,\end{eqnarray} for
all $a_1, \cdots, a_n \in \mathscr{A}$.
\end{theorem}
\begin{proof}
By using Lemma \ref{8}, we have
\begin{eqnarray*}
\left|\sum\limits_{i=1}^ny_ia_i\right|^2&=&\left\langle
\sum\limits_{i=1}^n
 y_ia_i,\sum\limits_{i=1}^n
y_ia_i\right\rangle=\sum\limits_{1\leq i,j\leq n} a_i^*\langle y_i,y_j\rangle a_j\\
&=&\sum\limits_{i=1}^n a_i^*\langle y_i,y_i\rangle
a_i+\sum\limits_{1\leq i\neq j\leq n} a_i^*\langle y_i,y_j\rangle
a_j\\
&=&\sum\limits_{i=1}^n a_i^*\langle y_i,y_i\rangle
a_i+\sum\limits_{1\leq i< j\leq n} \left(a_i^*\langle y_i,y_j\rangle
a_j+a_j^*\langle y_j,y_i\rangle a_i\right)\\
&\leq&\sum\limits_{i=1}^n \|\langle
y_i,y_i\rangle\|\,|a_i|^2+\sum\limits_{1\leq i< j\leq n} (\|\langle
y_i,y_j\rangle\|\,|a_i|^2+\|\langle
y_i,y_j\rangle\|\,|a_j|^2)\\
&=&\sum\limits_{1\leq i,j\leq n}\|\langle y_i,y_j\rangle\|\,|a_i|^2\\
&\leq&\sum\limits_{i=1}^n|a_i|^2\max\limits_{1\leq i\leq
n}\sum\limits_{j=1}^n \|\langle y_i,y_j\rangle\|\,.
\end{eqnarray*}
\end{proof}
%---------------------------------------------------------------------
From Theorem \ref{main} the following result of Bombieri type can be
obtained.
\begin{corollary}
Let $x, y_1, \cdots, y_n\in {\mathscr X}$ and $a_1, \cdots, a_n\in
{\mathscr A}$. Then
$$\left(\sum_{i=1}^n|\langle y_i,x\rangle|^2\right)^2\leq|x|^2\left\|\sum_{i=1}^n|\langle y_i,x\rangle
|^2\right\|\max_{1\leq i\leq n}\sum_{j=1}^n \|\langle
y_i,y_j\rangle\|\,.$$
\end{corollary}
\begin{proof}
By using inequality (3.2), we have
\begin{eqnarray*}
\left(\sum_{i=1}^n|\langle
y_i,x\rangle|^2\right)^2&=&\left|\sum_{i=1}^n
\langle y_i,x\rangle^*\langle y_i,x\rangle\right|^2\\
&=&\left|\left\langle\sum_{i=1}^n
 y_i\langle y_i,x\rangle,x\right\rangle\right|^2\\
 &\leq&\left\|\sum_{i=1}^n
 y_i\langle y_i,x\rangle\right\|^2|x|^2\\
 &\leq&|x|^2\left\|\sum_{i=1}^n| \langle y_i,x\rangle
|^2\right\|\max\limits_{1\leq i\leq n}\sum_{j=1}^n \|\langle
y_i,y_j\rangle\|\\
&\leq&|x|^2\left\|\sum_{i=1}^n|\langle y_i,x\rangle
|^2\right\|\max\limits_{1\leq i\leq n}\sum_{j=1}^n \|\langle
y_i,y_j\rangle\|\,.
\end{eqnarray*}
\end{proof}
%--------------------------------------------------------------------

Now we state some other applications of inequality (\ref{eq}) for
Hilbert space operators, although some of them can be deduced
directly. Recall that the space ${\mathbb
B}(\mathscr{H},\mathscr{K})$ of all bounded linear operators between
Hilbert spaces $\mathscr{H}$ and $\mathscr{K}$ can be regarded as a
Hilbert $C^*$-module over the $C^*$-algebra ${\mathbb B}(\mathscr
H)$ via $\langle T,S\rangle=T^*S$.
\begin{corollary}
Let $\mathscr {H}$ and $\mathscr{K}$ be Hilbert spaces and let $T_1,
\cdots, T_n \in {\mathbb B}(\mathscr{H},\mathscr{K})$ be operators
having orthogonal ranges. Then
$$\left|\sum_{i=1}^n T_iS_i\right|^2\leq \sum_{i=1}^n |S_i|^2  \max_{1\leq i \leq n} \|T_i\|^2,$$
for any $S_1,\cdots, S_n \in {\mathbb B}(\mathscr{H})$.
\end{corollary}
\begin{proof}
Clearly the operators $T_i$ have orthogonal ranges if and only if
$T_i^*T_j=0$. Thus the result follows immediately from inequality
(\ref{eq}).
\end{proof}
\begin{corollary}
Let $S_1,S_2$ be operators on a Hilbert space $\mathscr {H}$ and let
$T$ be an invertible operator on $\mathscr {H}$. Then
$$|TS_1+(T^*)^{-1}S_2|^2\leq (|S_1|^2+|S_2|^2)\left[1+\max(\|T\|^2,\|T^{-1}\|^2)\right].$$
\end{corollary}
The next result is a refinement of one of the inequalities presented
in Theorem 2.1 of \cite{dra}.
\begin{corollary}
If $\lambda_1,\cdots\lambda_n$ are complex numbers and
$T_1,\cdots,T_n$ are operators on a Hilbert space $\mathscr {H}$,
then
$$\left|\sum_{i=1}^n\lambda_i T_i\right|\leq\max_{1\leq i\leq n}|\lambda_i|\sum_{i=1}^n
|\lambda_i|\sum_{i=1}^n|T_i|^2\,.$$
\end{corollary}
\begin{proof}
If we consider $\lambda_i T_i=(\lambda_i I)T_i$, we get from (3.2)
that
\begin{eqnarray*}
\left|\sum_{i=1}^n\lambda_i T_i\right|
&\leq&\sum_{i=1}^n|T_i|^2\max\limits_{1\leq i\leq n}\sum_{j=1}^n
|\bar{\lambda_i}\lambda_j|\\
&=&\sum_{i=1}^n|T_i|^2\max\limits_{1\leq i\leq
n}|\lambda_i|\sum_{j=1}^n |\lambda_j|\\
&=&\max\limits_{1\leq i\leq n}|\lambda_i|\sum_{i=1}^n
|\lambda_i|\sum_{i=1}^n|T_i|^2.
\end{eqnarray*}
\end{proof}
%----------------------------------------------------------------------------------
The following theorem is similar to the
Mitrinovi\'{c}--Pe\v{c}ari\'{c}--Fink theorem in the Hilbert space
theory, with $\mathscr{A}$-valued norm instead of the usual norm.
%-------------------------------------------------------------------------------------
\begin{theorem}
If $x,y_1,\cdots,y_n$ are elements of a Hilbert $\mathscr{A}$-module
$\mathscr{X}$ and $a_1,\cdots,a_n$ are elements of $\mathscr{A}$,
then
\begin{eqnarray}\label{pecaric}
\left|\sum\limits_{i=1}^n a_i \langle
y_i,x\rangle\right|^2\leq|x|^2\sum\limits_{i=1}^n\|a_i\|^2\left[
\max_{1\leq i\leq n}\|y_i\|^2+\left(\sum_{1\leq i\neq j\leq
n}\|\langle y_i,y_j\rangle\|^2\right)^{1/2}\right].\end{eqnarray}
\end{theorem}
\begin{proof}
By the Cauchy--Schwarz inequality,
\begin{eqnarray}\label{1}
\left|\sum\limits_{i=1}^n a_i\langle y_i,x\rangle\right|^2=
\left|\langle\sum\limits_{i=1}^n y_ia_i^*,x\rangle\right|^2\leq
\left\|\sum\limits_{i=1}^n y_ia_i^*\right\|^2|x|^2.\end{eqnarray} We
also have
\begin{eqnarray}\label{2}
\left\|\sum\limits_{i=1}^ny_ia_i^*\right\|^2 &=&\left\|\left\langle
\sum\limits_{i=1}^n y_ia_i^*,\sum\limits_{j=1}^n
y_ja_j^*\right\rangle\right\|\nonumber\\
&=&\left\|\sum\limits_{1\leq i,j\leq n} a_i\langle y_i,y_j\rangle a_j^*\right\|\nonumber\\
&\leq&\sum\limits_{1\leq i,j\leq n} \|a_i\|\|\langle y_i,y_j\rangle\|\| a_j\|\nonumber\\
&=&\sum\limits_{i=1}^n \|a_i\|^2\| y_i\|^2 +\sum\limits_{1\leq i\neq
j\leq n}
\|a_i\|\|a_j\|\|\langle y_i,y_j\rangle\|\nonumber\\
&\leq&\max\limits_{1\leq i\leq
n}\|y_i\|\sum\limits_{i=1}^n\|a_i\|^2+\left(\sum\limits_{1\leq i\neq
j\leq n } \|a_i\|^2\|a_j\|^2\right)^{\frac{1}{2}}
\left(\sum\limits_{1\leq i\neq j\leq n} \|\langle
y_i,y_j\rangle\|^2\right)^{\frac{1}{2}}\nonumber\\
&\leq&\sum\limits_{i=1}^n\|a_i\|^2\left(\max\limits_{1\leq i\leq
n}\|y_i\|^2+\left(\sum\limits_{1\leq i\neq j\leq n}\|\langle
y_i,y_j\rangle\|^2\right)^{\frac{1}{2}}\right).
\end{eqnarray}
Combining (\ref{1}) and (\ref{2}) we can get the desired result.
\end{proof}
%------------------------------------------------------------------------------------
\begin{corollary}
For $x,y_1,\cdots,y_n$ in a Hilbert $\mathscr{A}$-module
$\mathscr{X}$,
\begin{eqnarray}\label{boas}
\begin{array}{l}
\left(\sum\limits_{i=1}^n|\langle y_i,x\rangle|^2\right)^2
\leq|x|^2\sum\limits_{i=1}^n\left\|\langle
y_i,x\rangle\right\|^2\left[\max\limits_{1\leq i\leq
n}\|y_i\|^2+\left( \sum\limits_{1\leq i\neq j\leq n } \|\langle
y_i,y_j\rangle\|^2\right)^{\frac{1}{2}}\right].\end{array}
\end{eqnarray}
\end{corollary}
\begin{proof}
Set $a_i=\langle x,y_i\rangle$ in inequality (\ref{pecaric}).
\end{proof}
%-------------------------------------------------------------------------------------
Note that the inequality (\ref{boas}), can be considered  as a
generalization of Boas--Bellman inequality (1.1). However, for the
case where $(y_i)_{1\leq i \leq n}$ is an orthonormal sequence of
vectors, the inequality (\ref{boas}) is a weaker result than
(\ref{bessel}).

Furthermore, all of the inequalities which are obtained by Dragomir
in \cite[Chapter 4]{drabook} can be extended to Hilbert
$C^*$-modules in a similar way. The details are left to the
interested readers.

We can also prove some other Boas--Bellman type inequalities in
Hilbert $C^*$-modules as follows.
%--------------------------------------------------------------------------------------
\begin{lemma}
Let $\mathscr{A}$ be a $C^*$-algebra and $\mathscr{X}$ be a Hilbert
$\mathscr{A}$-module. Then
$$\left|\sum\limits_{i=1}^n y_ia_i\right|^2\leq\max\limits_{1\leq
i\leq n} \|y_i\|^2\sum\limits_{i=1}^n|a_i|^2+B_n,$$ where
\begin{eqnarray}\label{mos}
B_n=\left\{\begin{array}{l} (n-1)\sqrt{n} \max\limits_{1\leq i\leq
n}\|a_i\|\max\limits_{1\leq i\neq j\leq n}\|\langle
y_i,y_j\rangle\|\left(\sum\limits_{i=1}^n|a_i|^2\right)^{\frac{1}{2}},\\
\sqrt{n-1}\left(\max\limits_{1\leq i\leq n} \sum\limits_{1\leq j\neq
i\leq n} \|\langle
y_i,y_j\rangle\|^2\right)^{\frac{1}{2}}\left(\sum\limits_{i=1}^n\|a_i\|^2\right)^{\frac{1}{2}}\left(
\sum\limits_{i=1}^n|a_i|^2\right)^{\frac{1}{2}},\end{array}\right.
\end{eqnarray} for any $y_i\in \mathscr{X}$ and $a_i\in
\mathscr{A}$.
\end{lemma}
\begin{proof} We observe that
\begin{eqnarray*}
|\sum\limits_{i=1}^ny_ia_i|^2&=&\langle \sum\limits_{i=1}^n
 y_ia_i,\sum\limits_{i=1}^n
y_ia_i\rangle\\
&=&\sum\limits_{1\leq i,j\leq n} a_i^*\langle y_i,y_j\rangle a_j\\
&=&\sum\limits_{i=1}^n a_i^*\langle y_i,y_i\rangle
a_i+\sum\limits_{1\leq i\neq j\leq n} a_i^*\langle y_i,y_j\rangle
a_j.
\end{eqnarray*}
We also have
$$\sum\limits_{i=1}^n a_i^*\langle y_i,y_i\rangle a_i=\sum\limits_{i=1}^n a_i^*| y_i|^2
a_i\leq \sum\limits_{i=1}^n\|y_i\|^2|a_i|^2\leq\max_{1\leq i\leq
n}\|y_i\|^2 \sum\limits_{i=1}^n|a_i|^2.$$
To get the first inequality
of \eqref{mos} note that
\begin{eqnarray*}
\sum\limits_{1\leq i\neq j\leq n} a_i^*\langle y_i,y_j\rangle
a_j&=&2\sum\limits_{1\leq i< j\leq n}
{\rm Re}(a_i^*\langle y_i,y_j\rangle a_j)\\
&\leq& 2\sum\limits_{1\leq i< j\leq n}\left(\frac{|a_i^*\langle
y_i,y_j\rangle
a_j|^2+|a_j^*\langle y_j,y_i\rangle a_i|^2}{2}\right)^{1/2}\\
&&\qquad\qquad\qquad \quad \big({\rm by~} |{\rm Re}(c)| \leq
\left(\frac{c^*c+cc^*}{2}\right)^{1/2}, {\rm ~where~} c \in {\mathscr A}\big)\\
&=&\sqrt{2}\sum\limits_{1\leq i< j\leq n}\left({|a_i^*\langle
y_i,y_j\rangle a_j|^2+|a_j^*\langle y_j,y_i\rangle
a_i|^2}\right)^{1/2}.
\end{eqnarray*}
By operator convexity of the function $f(t)=t^2$, we conclude that
\begin{eqnarray*}
\sum\limits_{i\neq j} a_i^*\langle y_i,y_j\rangle
a_j&\leq&\sqrt{2}.\sqrt{\frac{n^2-n}{2}}\left(\sum\limits_{1\leq i<
j\leq n}{|a_i^*\langle y_i,y_j\rangle
a_j|^2+|a_j^*\langle y_j,y_i\rangle a_i|^2}\right)^{1/2}\\
&=&\sqrt{n^2-n}\left(\sum\limits_{1\leq i\neq j\leq n}|a_i^*\langle
y_i,y_j\rangle a_j|^2
\right)^{1/2}\\
&\leq&\sqrt{n^2-n}\max\limits_{1\leq i\leq
n}\|a_i\|\max\limits_{1\leq i\neq j\leq n}\|\langle
y_i,y_j\rangle\|\left((n-1)\sum\limits_{i=1}^n|a_i|^2\right)^{1/2}\\
&=&(n-1)\sqrt{n}\max\limits_{1\leq i\leq n}\|a_i\|\max\limits_{1\leq
i\neq j\leq n}\|\langle
y_i,y_j\rangle\|\left(\sum\limits_{i=1}^n|a_i|^2\right)^{1/2}.\\
\end{eqnarray*}
To obtain the second inequality of \eqref{mos} notice that
$\sum\limits_{1\leq i\neq j\leq n} a_i^*\langle y_i,y_j\rangle a_j$
is self-adjoint, hence
\begin{eqnarray*}
\left(\sum\limits_{1\leq
i\neq j\leq n} a_i^*\langle y_i,y_j\rangle
a_j\right)^2 &\leq& \left|\sum\limits_{1\leq i\neq j\leq n} a_i^*\langle y_i,y_j\rangle a_j\right|^2\\
&=&\left|\sum\limits_{i=1}^na_i^*\left(\sum\limits_{1\leq j\neq
i\leq n} \langle y_i,y_j\rangle
a_j\right)\right|^2\\
&\leq&\sum\limits_{i=1}^n\|a_i\|^2\sum\limits_{i=1}^n\left|\sum\limits_{1\leq
j\neq i\leq n} \langle y_i,y_j\rangle
a_j\right|^2\\
&\leq&\sum\limits_{i=1}^n\|a_i\|^2\sum\limits_{i=1}^n\left(\sum\limits_{1\leq
j\neq i\leq n}\|\langle y_i,y_j\rangle\|^2
\sum\limits_{1\leq j\neq i\leq n}|a_j|^2\right)\\
&\leq&\sum\limits_{i=1}^n\|a_i\|^2\left(\max\limits_{1\leq i\leq
n}\sum\limits_{1\leq j\neq i\leq n} \|\langle y_i,y_j\rangle\|^2.
\sum\limits_{i=1}^n\sum\limits_{1\leq j\neq i\leq n}|a_j|^2\right)\\
&\leq&(n-1)\sum\limits_{i=1}^n|a_i|^2\left(\sum\limits_{i=1}^n\|a_i\|^2.\max\limits_{1\leq
i\leq n} \sum\limits_{1\leq j\neq i\leq n} \|\langle
y_i,y_j\rangle\|^2\right).
\end{eqnarray*}
\end{proof}
%--------------------------------------------------------------------------------------
The following result can be stated as well.
\begin{theorem}
Let $x,y_1,\cdots,y_n\in \mathscr{X}$ and $a_1,\cdots,a_n\in
\mathscr{A}$. Then
$$\left|\sum_{i=1}^n a_i\langle
y_i,x\rangle\right|^2\leq|x|^2\left\|\sum\limits_{i=1}^n|a_i^*|^2\right\|^{1/2}\left[\max\limits_{1\leq
i\leq n}
\|y_i\|^2\|\sum\limits_{i=1}^n|a_i^*|^2\|^{1/2}+B_n\right],$$ where
$$B_n:=\left\{\begin{array}{l}(n-1)\sqrt{n}\max\limits_{1\leq
i\leq n}\|a_i\|\max\limits_{1\leq i\neq j\leq n}\|\langle
y_i,y_j\rangle\|,\\
\sqrt{n-1}\left(\max\limits_{1\leq i\leq n}\sum\limits_{1\leq j\neq
i\leq n}\|\langle
y_i,y_j\rangle\|^2\right)^{1/2}\left(\sum\limits_{i=1}^n\|a_i\|^2\right)^{1/2}.\end{array}\right.$$
\end{theorem}
\begin{remark}
In the case where $(y_i)$ is an orthogonal sequence, it follows from
the theorem above that
$$\left(\sum\limits_{i=1}^n|\langle y_i,x\rangle|^2\right)^2\leq|x|^2\max\limits_{1\leq i\leq n}
\|y_i\|^2\left\|\sum\limits_{i=1}^n \left|\langle
y_i,x\rangle\right|^2\right\|,$$ which is a stronger result than the
inequality (\ref{boas}).\end{remark}

{\bf Acknowledgement.}

This work was done when the second author was at the Research Group
in Mathematical Inequalities and Applications (RGMIA) in Victoria
University on her sabbatical leave from Tehran Teacher Training
University. She thanks both universities for their support. The
third author was supported by a grant from Ferdowsi University of
Mashhad; (NO. MP88060MOS)
%--------------------------------------------------------------------------------------------------%


\begin{thebibliography}{10}
\bibitem{bellman}
R. Bellman, \textit{Almost orthogonal series}, Bull. Amer. Math.
Soc. \textbf{50} (1944), 517--519.

\bibitem{boas}
R.P. Boas, \textit{A general moment problem}, Amer. J. Math.
\textbf{63} (1941), 361--370.

\bibitem{bom}
E. Bombieri, \textit{A note on the large sieve}, Acta Arith.
\textbf{18} (1971), 401--404.

\bibitem{drabook}
S.S. Dragomir, \textit{Advances in Inequalities of the Schwarz,
Gr$\ddot{u}$ss and Bessel Type in Inner Product Spaces}, Nova
Science Publishers, Inc., Hauppauge, NY, 2005.

\bibitem{dra}
S.S. Dragomir, \textit{Norm inequalities for sequences of operators
related to the Schwarz inequality}, J. Inequal. Pure Appl. Math.
\textbf{7} (2006), no. 3, Article 97.

\bibitem{F-L} M. Frank and D.R. Larson, \textit{Frames in Hilbert $C\sp
\ast$-modules and $C\sp \ast$-algebras}, J. Operator Theory
\textbf{48} (2002), no. 2, 273--314.

\bibitem{lance}
E.C. Lance, \textit{Hilbert $C^*$-modules}, London Mathematical
Society Lecture Note Series, 210, Cambridge University Press,
Cambridge, 1995.

\bibitem{M-T} V.M. Manuilov and E.V. Troitsky, \textit{Hilbert $C^*$-modules},
Translations of Mathematical Monographs, 226. American Mathematical
Society, Providence, RI, 2005.

\bibitem {pec} D.S. Mitrinovi\'{c}, J.E. Pe\v{c}ari\'{c} and A.M. Fink, \textit{Classical and New Inequalities in Analysis},
Kluwer Academic, Dordrecht, 1993.

\end{thebibliography}
\end{document}